  \documentclass[letterpaper,11pt,reqno]{amsart}

  \usepackage[utf8]{inputenc}
  \usepackage{etex}
  \usepackage{xcolor}

  \usepackage{todonotes}

  \definecolor{verydarkblue}{rgb}{0,0,0.5}

  \usepackage[
      breaklinks,
      colorlinks,
      citecolor=verydarkblue,
      linkcolor=verydarkblue,
      urlcolor=verydarkblue,
      pagebackref=true,
      hyperindex
  ]{hyperref}

  \backrefenglish

  \usepackage{fancyhdr}

  \usepackage[
      hscale=0.7,
      vscale=0.75,
      headheight=13pt,
      centering,
  ]{geometry}

  \usepackage{amsfonts}
  \usepackage{amsmath}
  \usepackage{amssymb}
  \usepackage{amsthm}
  \usepackage{mathtools}
  \usepackage{stmaryrd}
  \usepackage{mathdots}

  \usepackage{url}
  \usepackage[all]{xy}
  \usepackage{latexsym}
  \usepackage{graphicx}

  \newtheorem{theorem}{Theorem}[section]		
  \newtheorem{lemma}[theorem]{Lemma}
  			
  \newtheorem{prop}[theorem]{Proposition}
  \newtheorem{rmk}[theorem]{Remark}				
  
  \newtheorem*{defi*}{Definition}			
  \newtheorem*{bei*}{Example}
  \newtheorem*{sat*}{Theorem}				
  \newtheorem*{kor*}{Corollary}
  	
  \newtheorem*{quest*}{Question}	
  \newtheorem{fact}{Fact}	
  \newtheorem*{claim}{Claim}	

  \let\ssection=\section
  \renewcommand{\section}{\setcounter{equation}{0}\ssection}

  \newtheorem*{namedtheorem}{\theoremname}
  \newcommand{\theoremname}{testing}
  \newenvironment{named}[1]{\renewcommand{\theoremname}{#1}\begin{namedtheorem}}{\end{namedtheorem}}

  \theoremstyle{remark}
  \newtheorem*{bem}{Remark}

  \newtheorem*{namedtheoremr}{\theoremnamer}
  \newcommand{\theoremnamer}{testing}

  \IfFileExists{./article-style.tex}{\input{article-style.tex}}{}

  \newcommand{\BR}{\mathbb R}			
  			
  			\newcommand{\BZ}{\mathbb Z}

  \newcommand{\RP}{\mathrm P}

  \newcommand{\CC}{\mathcal C}		\newcommand{\calD}{\mathcal D}
  		\newcommand{\CF}{\mathcal F}
  \newcommand{\CG}{\mathcal G}		
  \newcommand{\CI}{\mathcal I}		
  		\newcommand{\CL}{\mathcal L}
  \newcommand{\CM}{\mathcal M}		\newcommand{\CN}{\mathcal N}
  		
  		\newcommand{\CR}{\mathcal R}

  		\newcommand{\CZ}{\mathcal Z}

  \newcommand{\actson}{\curvearrowright}
  \newcommand{\D}{\partial}

  \DeclareMathOperator{\Diff}{Diff}	%	Diffeomorphimen einer Mf
  \DeclareMathOperator{\SL}{SL}		%	Spezielle lineare Gruppe
  \DeclareMathOperator{\Map}{Map}

  \DeclareMathOperator{\rel}{rel}

  \newcommand{\comment}[1]{}

  \DeclareMathOperator{\const}{const}

  \DeclareMathOperator{\Lip}{Lip}
  \DeclareMathOperator{\Thu}{Thu}

\begin{document}

%section{Title and abstract}

  \title[Genericity of pseudo-Anosovs]{Genericity of pseudo-Anosov mapping
  classes,\\ when seen as mapping classes}
  \author{Viveka Erlandsson}
  \address{School of Mathematics, University of Bristol \\ Bristol BS8 1UG, UK {\rm and}  \newline ${ }$ \hspace{0.2cm} Department of Mathematics and Statistics, UiT The Arctic University of Norway}
  \email{v.erlandsson@bristol.ac.uk}

  \author{Juan Souto}
  \address{UNIV RENNES\\ CNRS\\ IRMAR - UMR 6625\\ F-35000 RENNES\\ FRANCE}
  \email{juan.souto@univ-rennes1.fr}

  \author{Jing Tao}
  \address{Department of Mathematics, University of Oklahoma \\ 601 Elm
  Avenue Room 423 \\ Norman, OK 73069}
  \email{jing@ou.edu}

  \begin{abstract}
    
    We prove that pseudo-Anosov mapping classes are generic with respect to
    certain notions of genericity reflecting that we are dealing with
    mapping classes. More precisely, we consider a number of functions
    $\rho$ on the mapping class group, and show that the proportion of
    pseudo-Anosov mapping classes with $\rho$--value at most $R$ tends to
    $1$ as $R$ tends to infinity. The functions we consider include
    measuring the complexity of mapping classes using quasi-conformal
    distortion or Lipschitz distortion. We present a uniform approach to
    this problem using geodesic currents.  

  \end{abstract}
  
  \thanks{The first author was partially supported by EPSRC grant
  EP/T015926/1.\\ \indent The third author was partially supported by NSF
  DMS-1611758 and DMS-1651963. }

  \maketitle

\section{Introduction}

  Throughout this paper let $\Sigma$ be a complete orientable hyperbolic
  surface of finite area, with genus $g$ and $r$ punctures. We exclude the
  case of $(g,r)\neq(0,3)$ so that the mapping class group $\Map(\Sigma)$
  of $\Sigma$ is infinite.  

  Thurston's classification asserts that elements in $\Map(\Sigma)$ fall
  into three categories: finite order, reducible, and pseudo-Anosov.
  However, it seems that from any reasonable point of view most elements
  are pseudo-Anosov. For example, Maher \cite{Maher1} proved that, with few
  assumptions, random walks on the mapping class group give rise to
  pseudo-Anosov elements with asymptotic probability one. This result was
  later enhanced and generalized by Maher himself and others  \cite{Maher3,
  Maher-Tiozzo, Rivin, Sisto, Taylor-Tiozo}. 

  We will however care about another notion of genericity: if
  $\rho:\Map(\Sigma)\to\BR_{\ge 0}$ is a proper positive function, then we
  say that a set $X\subset\Map(\Sigma)$ is {\em generic} with respect to
  $\rho$, or {\em $\rho$-generic} for short, if we have
  $$\lim_{R\to\infty}\frac{\vert B^\rho(R)\cap X\vert}{\vert
  B^\rho(R)\vert}=1$$ where $B^\rho(R)=\{\phi\in\Map(\Sigma)\text{ with
  }\rho(\phi)\le R\}$. Here properness of $\rho$ just means that
  $B^\rho(R)$ is a finite set for all $R$. A {\em negligible} set is one
  whose complement is generic.

  Maybe the first function that comes to mind is the word length with
  respect to a finite generating set $\CG$ of $\Map(\Sigma)$, and Cumplido
  and Wiest \cite{Cumplido-Wiest} proved that indeed the set of
  pseudo-Anosov elements is not negligible in this sense. For braid groups
  equipped with the Garside's generating set, Caruso and Wiest
  \cite{Caruso-Wiest} showed that it is also generic. But beyond this case,
  genericity of pseudo-Anosov elements remain open for word lengths. 

  However, one can make the case that the word length, while being related
  to the group theory of the mapping class group, has little to do with the
  fact that the mapping class group consists of mapping classes. To
  illustrate this point identify $\SL_2\BZ$ with the mapping class group of
  the once punctured torus and note that the two matrices
  $$
  M=\left(
  \begin{array}{cc}
  5904283700961130691 & 4322235651404355330\\
  2161117825702177665 & 1582048049556775361
  \end{array}
  \right),\ 
  N=\left(
  \begin{array}{cc}
  1 & 99\\
  0 & 1
  \end{array}
  \right)$$ have the same word length, namely 99, with respect to the
  standard generating set of $\SL_2\BZ$. 
  %We tend to suspect that not only simpletons like ourselves, but also
  %really sophisticated scholars like the reader, would find more 
  Arguably, it would be more natural to say that $M$ is farther from the
  identity than $N$. Not only because the coefficients of $M$ are much
  larger than those of $N$ but, more importantly, because the map induced
  by $M$ on the torus distorts both the metric and conformal structure much
  more dramatically than the map induced by $N$. 

  Our goal is to prove that pseudo-Anosov mapping classes are
  $\rho$-generic with respect to a number of functions on $\Map(\Sigma)$
  measuring the complexity of mapping classes when seen as mapping classes.
  More precisely, given $f \in \Diff(\Sigma)$, denote by $K(f)$ the
  quasi-conformal distortion of $f$, and by $\Lip(f)$ the Lipschitz
  constant of $f$. Also denote by $\iota(\cdot,\cdot)$ the geometric
  intersection number between two multicurves. We show: 

  \begin{theorem}\label{thm-distances}
    The set of pseudo-Anosov mapping classes is generic with respect to any
    one of the functions:
    \begin{enumerate}
      \item $\rho_{K}(\phi)=\inf\{K(f)\text{ where }f\in\Diff(\Sigma)\text{
      represents }\phi\}$. 
      \item $\rho_{\Lip}(\phi)=\inf\{\Lip(f)\text{ where }f\in\Diff(\Sigma)\text{
      represents }\phi\}$.
      \item $\rho_{\sigma,\eta}(\phi)=\iota(\phi(\sigma),\eta)$, where
        $\sigma$ and $\eta$ are filling multicurves. 
    \end{enumerate}
  \end{theorem}

  \begin{bem}
    Note that, although amazingly it is not formally stated in the paper, the
    claim for $\rho_K(\phi)$ in Theorem \ref{thm-distances} was obtained by
    Maher in \cite{Maher2}. Unfortunately, we were unaware of this fact until
    we finished writing our paper. Both the argument in \cite{Maher2} and
    ours have the same starting point, namely an earlier, again not formally
    stated, result from \cite{Maher1}. However, after that starting point,
    the arguments use different methods and techniques. We will return to
    this at the end of the introduction. 
  \end{bem}  
  
  We sketch now the proof of Theorem \ref{thm-distances}. We begin by
  addressing the reason why we are including $\rho_{\sigma,\eta}$ at all
  among the functions in Theorem \ref{thm-distances}. There are a few
  reasons. First, both quantities $\rho_K(\phi)$ and $\rho_{\Lip}(\phi)$
  can be estimated in terms of $\rho_{\sigma,\eta}$. Second, there is the
  maybe not very important observation that, after identifying $\SL_2\BZ$
  with the mapping class group of a punctured torus, the $\ell_1$-norm on
  $\SL_2\BZ$ agrees with $\rho_{\sigma,\sigma}$ where $\sigma$ is the union
  of the two simple curves representing the standard generators of
  homology. However, the main reason to consider $\rho_{\sigma,\eta}$ is
  that it is the more natural quantity from the point of view of proofs. 

  In fact, if we denote by $\CC(\Sigma)$ the space of geodesic currents on
  $\Sigma$ endowed with the weak-* topology, and consider multicurves as
  currents, then what we will actually prove is the following theorem:
  %$\CR=\{\phi\in\Map(\Sigma)\text{ non-pseudo-Anosov}\}$
  \begin{theorem}\label{thm-counting curves}
    Let $\CR \subset \Map(\Sigma)$ be the set of non-pseudo-Anosov mapping
    classes and let $\gamma_0\subset\Sigma$ be a filling multicurve. Then
    we have $$\lim_{L\to\infty}\frac{\vert\{\phi\in\CR\text{ with
    }F(\phi(\gamma_0))\le L\}\vert}{L^{6g-6+2r}}=0$$ for every continuous
    homogenous function $F:\CC(\Sigma)\to\BR_{\ge 0}$ which, for every
    compact $A\subset\Sigma$, is proper when restricted on the set
    $\CC_A(\Sigma)$ of currents supported by $A$.
  \end{theorem}

  Recall that a function $F:\CC(\Sigma)\to\BR$ is {\em homogenous} if
  $F(t\cdot\lambda)=t\cdot F(\lambda)$ for every $t\ge 0$ and
  $\lambda\in\CC(\Sigma)$. Note also that for $\Sigma$ open, the properness
  condition we impose on $F$ is much weaker it being proper on
  $\CC(\Sigma)$. For example, if $\eta$ is a filling multicurve then
  $F(\cdot)=\iota(\cdot,\eta)$ is not proper on $\CC(\Sigma)$ but is proper
  on $\CC_A(\Sigma)$ for any $A$. Theorem \ref{thm-distances} follows when
  we apply Theorem \ref{thm-counting curves} to the corresponding functions
  combined with the fact, see \cite{ES-GAFA,Sapir}, that
  \begin{equation}\label{eq-maryam}
    \liminf_{L\to\infty}\frac{\vert\{\phi\in\Map(\Sigma)\text{ with
    }F(\phi(\gamma_0))\le L\}\vert}{L^{6g-6+2r}}>0
  \end{equation}
  for any $F$ as in Theorem \ref{thm-counting curves}. 

  %  Note that \eqref{eq-maryam} is nothing but a coarse version of a result
  %  of Mirzakhani \cite{Maryam1,Maryam2} and in the proof of Theorem
  %  \ref{thm-counting curves} we will make use of some of the arguments used
  %  by the first two authors \cite{ES-GAFA,ES-announcement,Book} to reprove
  %  Mirzakhani's theorem.

  The starting point of the proof of Theorem \ref{thm-counting curves} is a
  result of Maher \cite{Maher1} asserting that the set
  $\CR\subset\Map(\Sigma)$ of non-pseudo-Anosov mapping classes is the
  union, for each $k$, of $k$-isolated points (that is, points which are at
  distance at least $k$ from any other element of $\CR$) together with the
  union of finitely many sets, each one of which consists of mapping
  classes at {\em relative distance} $L(k)$ around the centralizer of some
  mapping class. Here the relative distance is the semi-distance on
  $\Map(\Sigma)$ arising, with the help of a base point, from its action on the
  curve complex. It follows that proving that $\CR$ is negligible boils
  down to proving (1) that the
  set $\CI_k\subset\CR$ of $k$-isolated points has low density and (2) that
  sets of mapping classes with small relative distance of centralizers of
  elements are negligible. Rephrasing this in terms of measures (on the
  space of currents) it suffices to prove (1) that 
  \begin{equation}\label{eq-comiendo roscon}
    \lim_{k\to\infty}\lim_{L\to\infty}\frac
    1{L^{6g-6+2r}}\sum_{\phi\in\CI_k}\delta_{\frac 1L\phi(\gamma_0)}=0,
      \end{equation}
  and (2) that 
  \begin{equation}\label{eq-comiendo roscon2}
    \lim_{L\to\infty}\frac
    1{L^{6g-6+2r}}\sum_{\phi\in\CN_{\rel}(C(\phi_0),R)}\delta_{\frac
    1L\phi(\gamma_0)}=0
  \end{equation}
  for $\phi_0\in\Map(\Sigma)$ non-central. Here $\delta_x$ is the Dirac
  measure centred on $x$ and the convergence takes place with respect to
  the weak-*-topology. We get \eqref{eq-comiendo roscon2} from the fact
  that any limit is absolutely continuous to the Thurston measure --- an
  immediate consequence of for example Proposition 4.1 in \cite{ES-GAFA}
  --- and of the fact that the set of limits of sequences of the form
  $(\phi_i(\gamma_0))$ with $\phi\in\CN_{\rel}(C(\phi_0),R)$ has vanishing
  Thurston measure. To establish \eqref{eq-comiendo roscon} we use again
  that any limit is absolutely continuous with respect to the Thurston
  measure, but this time we have to use Masur's result \cite{Masur} on the
  ergodicity of the Thurston measure with respect to the action of the
  mapping class group. 
    
  \begin{bem}
    Maher's proof in \cite{Maher2} of Theorem \ref{thm-distances} also
    relies on the decomposition of $\CR$ as the union of $\CI_k$ and
    finitely many sets consisting of mapping classes at bounded relative
    distance from the centralizer of some mapping class. At this point the
    two arguments diverge. While we rely on the fact that every limit of
    \eqref{eq-comiendo roscon} and \eqref{eq-comiendo roscon2} is
    absolutely continuous with respect to the Thurston measure, Maher makes
    use of a rather sophisticated lattice counting result of
    Athreya-Bufetov-Eskin-Mirzakhani \cite{ABEM}. Similarly, while we rely
    on the ergodicity of the Thurston measure, that is the ergodicity of
    the Teichm\"uller flow, Maher relies on the mixing property of that
    flow. We might be partial, but we believe that our argument is not only
    different but also simpler than that of Maher.

  \end{bem}

  \begin{bem}
    As it is the case for Maher's argument, all the results here hold with
    unchanged proofs if we replace the set $\CR$ of non-pseudo Anosov
    elements by any set of elements for which there is a uniform upper
    bound for the translation length in the curve complex.
  \end{bem}

  \subsection*{Acknowledgments}
  
  The last two authors started discussing the issues treated here while
  visiting the Fields Institute in the framework of the thematic program
  {\em Teichm\"uller Theory and its Connections to Geometry, Topology and
  Dynamics}. We are very thankful for the support of the Fields Institute
  and for many discussions with Kasra Rafi. The last author would also like
  to thank the second author and Anna Lenzhen for their hospitality last
  summer during which time this project was discussed. 

\section{Maher's theorem}

  As we already did in the introduction, we denote by $\CR$ the set of all
  non-pseudo-Anosov mapping classes of $\Map(\Sigma)$. We also fix an
  arbitrary finite generating set $\CG$ for $\Map(\Sigma)$ and let $d_\CG$
  be the induced left-invariant distance: $$d_{\CG}(\phi,\psi)=\text{word
  length with respect to }\CG\text{ of }\psi^{-1}\phi.$$ Given $k>0$ let
  $$\CI_k=\{\phi\in\CR\text{ with }d_\CG(\phi,\phi')\ge k\text{ for all
  }\phi'\in\CR\setminus\{\phi\}\}$$ be the set of elements in $\CR$ which
  do not have any other elements in $\CR$ within distance less than $k$. We
  denote the complement of $\CI_k$ by $$\calD_k=\CR\setminus\CI_k.$$ The
  notations are chosen to suggest that $\CI_k$ consists of {\em
  $k$-isolated points} and that $\calD_k$ consists of {\em $k$-dense
  points}. 

  Recall that distances in the definition of $\CI_k$ (and thus in that of
  $\calD_k$ as well) are measured with respect to the distance $d_{\CG}$.
  We stress that this is the case because we will also be working with
  another distance, or rather a semi-distance, namely the {\em relative
  distance}
  $$d_{\rel}(\phi,\psi)=d_{C(\Sigma)}(\phi(\alpha_0),\psi(\alpha_0))$$
  where $d_{C(\Sigma)}(\cdot, \cdot)$ denotes the distance in the curve
  complex $C(\Sigma)$, and where $\alpha_0$ is a fixed but otherwise
  arbitrary simple essential curve in $\Sigma$. Note that
  $d_{\rel}(\cdot,\cdot)$ is not a proper metric. 

  Armed with this notation we can state Maher's theorem:

  \begin{theorem}[Maher]\label{thm-maher}
    For every $k$, there is a finite set of non-central mapping classes
    $\CF\subset\Map(\Sigma) \setminus C(\Map(\Sigma))$ and some $L>0$ such
    that
    $$\calD_k\subset\bigcup_{\phi\in\CF}\{\psi\in\Map(\Sigma)\text{ with
    }d_{\rel}(\psi,C(\phi))\le L\},$$ where $C(\phi)$ is the centralizer of
    $\phi$ in $\Map(\Sigma)$ and $C(\Map(\Sigma))$ is the center of
    $\Map(\Sigma)$.
  \end{theorem}

  We remark that $\Map(\Sigma_{g,r})$ has trivial center if $(g,r) \notin
  \{(1,1), (1,2), (2,0)\}$. Although it is proved and used in
  \cite{Maher1} (see the discussion at the beginning of section 5 in
  said paper), Theorem \ref{thm-maher} is not explicitly stated therein.
  Hence we discuss how to deduce it from the stated results here:

  \begin{proof}
    First, suppose that $\Map(\Sigma)$ is center free. Then, from the very
    definition of $\calD_k$, we get that there is a finite subset
    $\CF\subset\Map(\Sigma)$ with
    \begin{equation}\label{eq-no idea what name}
    \calD_k\subset\bigcup_{\phi\in\CF}(\CR\cap\CR\phi).
    \end{equation}
    To see this, note that one can take $\CF$ to be all non-trivial elements in the ball of
    radius $k$ around the identity with respect to $d_\CG$.

    Now, Theorem 4.1 in \cite{Maher1} implies that for each $\phi\in\CF$
    there is some $L$ such that
    $$\CR\cap\CR\phi\subset\{\psi\in\Map(\Sigma)\text{ with }d_{\rel}(\psi,C(\phi))\le L\}.$$
    This theorem applies because the mapping class group is weakly
    relatively hyperbolic with relative conjugacy bounds \cite[Theorem
    3.1]{Maher1} and because $\CR$ consists of elements conjugated to
    elements of bounded relative length \cite[Lemma 5.5]{Maher1}. This
    concludes the discussion of Theorem \ref{thm-maher} if $\Map(\Sigma)$
    is center free. 

    In the presence of a non-trivial center the argument is almost the same: Note that $\CR=\CR\phi$
    for every central element and hence the only change to the above argument is that one has to
    take $\CF$ to be the set of all non-central elements in the ball of
    radius $k$ around the identity with respect to $d_\CG$.
  \end{proof}

\section{Currents}

  In this section we recall a few facts about the space of geodesic
  currents on $\Sigma$. We then describe the (projective) accumulation
  points of sequences of the form $(\phi_i(\gamma_0))$ where $\gamma_0$ is
  an essential multicurve and where $(\phi_i)$ is a sequence of mapping
  classes at bounded relative distance of the centralizer of some
  $\phi\in\Map(\Sigma)$. Recall that a multicurve is a finite union of
  (disjoint or not, simple or not) primitive essential curves in
  $\Sigma$. We say that a multicurve is filling if its geodesic
  representative cuts the surface into a collection of disks and
  once-punctured disks. 

  \subsection*{Properties of the space of currents}

  Let $\overline\Sigma$ be a compact surface with interior
  $\Sigma=\overline\Sigma\setminus\D\overline\Sigma$, endowed with an
  arbitrary hyperbolic metric with totally geodesic boundary. We suggest
  the reader to think, in a first reading, that $\overline\Sigma=\Sigma$;
  that is, $\Sigma$ is closed.

  Geodesic currents on $\Sigma$ are fundamental group invariant Radon
  measures on the space of geodesics on the universal cover of
  $\overline\Sigma$. However, that they are such measures will not really be
  relevant here---what is more important for our purposes 
  are the properties the space $\CC(\Sigma)$ of currents have (when
  endowed with the weak-*-topology). We list the facts about $\CC(\Sigma)$
  that we will use:
  \begin{enumerate}
    \item $\CC(\Sigma)$ is a locally compact metrizable topological space.
    \item $\CC(\Sigma)$ is a cone as a topological vector space, meaning in
      particular that there are continuous maps
    \begin{align*}
      \CC(\Sigma)\times\CC(\Sigma)&\to\CC(\Sigma),\ (\lambda,\mu)\mapsto\lambda+\mu\\
      \BR_{\ge 0}\times\CC(\Sigma)&\to\CC(\Sigma),\ (t,\lambda)\mapsto t\lambda
    \end{align*}
      satisfying the usual associativity, commutativity and distributivity
      properties as in vector spaces.
    \item The set $\{\gamma\text{ closed geodesic in }\Sigma\}$ is a subset
      of $\CC(\Sigma)$ and in fact the set \[\BR_+\cdot\{\gamma\text{
        closed geodesic in }\Sigma\}\] of weighted closed geodesics is
      dense in $\CC(\Sigma)$.
    \item The inclusion of the set of weighted simple geodesics into
      $\CC(\Sigma)$ extends to a continuous embedding of the space
      $\CM\CL(\Sigma)$ of measured laminations into $\CC(\Sigma)$.
    %\item There is sequences of closed geodesics $(\gamma_i)$ in $\Sigma$
      %and $(\epsilon_i)$ of positive numbers tending to $0$ with
      %$$\ell_\Sigma(\gamma)=\lim_{i\to\infty}\epsilon_i\cdot\iota(\gamma,\gamma_i)$$
      %for every curve $\gamma$ in $\Sigma$. Here $\ell_\Sigma(\cdot)$
      %stands for the hyperbolic length. 
    \item There is a continuous bilinear map
      $$\iota:\CC(\Sigma)\times\CC(\Sigma)\to\BR_{\ge 0}$$ such that
      $\iota(\gamma,\gamma')$ is nothing other than the geometric
      intersection number for all closed geodesics $\gamma,\gamma'$.
    \item The mapping class group acts continuously on $\CC(\Sigma)$ by
      linear automorphisms. Moreover, the inclusion of the set of closed
      geodesics into $\CC(\Sigma)$ is equivariant with respect to this
      action.
  \end{enumerate}
  Moreover, for every compact $A\subset\Sigma$, let
  $\CC_A(\Sigma)\subset\CC(\Sigma)$ be the subcone consisting of the
  currents supported by $A$. Then the following holds:
  \begin{enumerate}\setcounter{enumi}{6}
    \item The set $\{\lambda\in\CC_A(\Sigma)\text{ with
      }\iota(\lambda,\eta)\le L\}$ is compact for every $L\ge 0$ and every
      filling multicurve $\eta$. In particular, the image
      $\RP\CC_A(\Sigma)$ of $\CC_A(\Sigma)$ in the space
      $$\RP\CC(\Sigma)=(\CC(\Sigma)\setminus\{0\})/\BR_{>0}$$ of projective
      currents is compact. 
    \item For every multicurve $\gamma_0$ there is a compact
      $A\subset\Sigma$ such that
      $$\Map(\Sigma)\cdot\gamma_0\subset\CC_A(\Sigma).$$
      In particular, every sequence $(\phi_i)$ in $\Map(\Sigma)$ contains a
      subsequence $(\phi_{i_j})$ such that the limit
      $\lim_{j\to\infty}\phi_{i_j}(\gamma_0)$ exists in $\RP\CC(\Sigma)$.
    %\item For any two multicurves $\eta,\eta'$ with $\eta$ filling there
      %is a constant $L=L(\eta,\eta',K)$ with $$\iota(\lambda,\eta')\le
      %L\cdot \iota(\lambda,\eta)$$ for all $\lambda\in\CC_A(\Sigma)$.
    %\item The convergence in (5) is uniform for $\gamma$ in
      %$\CC_A(\Sigma)$. In particular, the length function
      %$\ell_\Sigma(\cdot)$ extends continuously to a proper homogenous
      %function on $\CC_A(\Sigma)$.
  \end{enumerate}
  Currents were introduced by Bonahon in \cite{Bonahon86,Bonahon90} and all
  the facts here can be found in a more or less transparent way in these
  papers. In the case of closed surfaces, \cite{Javi-Chris} is a very
  readable account of currents, measured laminations, and the relation
  between them. Finally, we hope that the presentation of currents, for
  both open and closed surfaces, in the forthcoming book \cite{Book} will
  also be similarly readable. 

  \subsection*{Accumulation points of thickened centralizers}

  It will be important later on to know that projective accumulation
  points, in the space of currents, of sequences of the form
  $(\phi_i(\gamma_0))$ where $\gamma_0$ is a multicurve and with
  $$\phi_i\in\CN_{\rel}(C(\phi),L)=\{\psi\in\Map(\Sigma)\text{ with
  }d_{\rel}(\psi,C(\phi))\le L\}$$ are very particular:

  \begin{prop}\label{kor silly kor}
    Let $\phi\in\Map(\Sigma)\setminus C(\Map(\Sigma))$ be a non-central
    mapping class, let $(\phi_n)$ be a sequence of pairwise distinct
    elements in $\CN_{\rel}(C(\phi),L)$, and let $\gamma_0$ be a filling
    multicurve. If the sequence $(\phi_n(\gamma_0))$ converges projectively
    to a uniquely ergodic measured lamination $\lambda$, then
    $\phi(\lambda)$ is a multiple of $\lambda$. 
  \end{prop}

  Recall that a measure lamination $\lambda$ is {\em uniquely ergodic} if
  every measured lamination $\mu$ with $\iota(\lambda,\mu)=0$ is a multiple
  of $\lambda$. 

  The proof of the proposition will make use of the following lemma. Note
  that the lemma requires that $\Sigma$ is not the once-punctured torus or
  the $4$-times punctured sphere. After the proof of the lemma, we will
  explain how to deal with those cases as well (see Remark \ref{rem
  silly}). 

  \begin{lemma}\label{lem silly lemma}
    Suppose $\Sigma$ is not the once-punctured torus or the $4$-times
    punctured sphere. 
    Let $\gamma_0\subset\Sigma$ be a filling multicurve
    and $(\phi_n)$ and $(\psi_n)$ be sequences of mapping classes with
    $d_{\rel}(\phi_n,\psi_n)\le L$. Given any simple multicurve $\alpha$,
    suppose that the sequences $(\phi_n(\gamma_0))$ and $(\psi_n(\alpha))$
    converge projectively to $\lambda,\lambda'\in P\CC(\Sigma)$,
    respectively. If $(\phi_n)$ consists of pairwise distinct elements,
    then there is a chain
    $$\lambda=\lambda_0,\lambda_1,\dots,\lambda_k=\lambda'$$ of measured
    laminations with $\iota(\lambda_i,\lambda_{i+1})=0$ for all
    $i=0,\dots,k-1$.
  \end{lemma}

  \begin{proof}
    Assume that $(\phi_n(\gamma_0))$ and $(\psi_n(\alpha))$ converge
    projectively to $\lambda, \lambda' \in P\CC(\Sigma)$. Abusing notation
    consider $\lambda$ and $\lambda'$ not only as projective currents but
    also as actual currents. The assumption that the sequences
    $(\phi_n(\gamma_0))$ and $(\psi_n(\alpha))$ converge projectively to
    $\lambda,\lambda'\in P\CC(\Sigma)$ implies that there are bounded
    sequences $(\epsilon_n)$ and $(\epsilon_n')$ consisting of positive
    numbers and such that $$\lambda=\lim_n\epsilon_n\phi_n(\gamma_0),\
    \lambda'=\lim_n\epsilon_n'\psi_n(\alpha).$$ The assumptions that
    $(\phi_n)$ consists of pairwise distinct elements and that $\gamma_0$
    is filling implies that the sequence $(\phi_n(\gamma_0))$ is not
    eventually constant, and thus that $\epsilon_n\to 0$.

    Let $\alpha_0$ be the base point in $C(\Sigma)$ used to define
    $d_{\rel}$. The assumption that $d_{\rel}(\phi_n,\psi_n) \le L$ implies
    that for all $n$ there is a chain of simple multicurves
    $$\phi_n(\alpha_0)=\beta_n^1,\beta_n^2,\dots,\beta_n^{L+1}=\psi_n(\alpha_0)$$
    with $\iota(\beta_n^i,\beta_n^{i+1})=0$ for all $i=1,\dots,L$ and all
    $n$. Also, there is a chain of simple multicurves \[\alpha_0, \alpha_1,
    \ldots, \alpha_m = \alpha\] with $\iota(\alpha_i,\alpha_{i+1}) =0$ for
    all $i=0,\ldots,m-1$. By setting $\beta_n^{L+1+i} = \psi_n(\alpha_i)$
    and $k=L+1+m$, we get a chain of simple multicurves 
    \[
      \phi_n(\alpha_0)=\beta_n^1,\beta_n^2,\dots,\beta_n^k=\psi_n(\alpha)
    \]
    with for $\iota(\beta_n^i,\beta_n^{i+1})=0$ for all $i=1,\dots,k-1$ and
    all $n$. 

    Projective compactness of the space of currents (or rather of measured
    laminations) implies that passing to a subsequence we may assume that
    there are bounded positive sequences
    $(\epsilon^1_n),\dots,(\epsilon^k_n)$ such that
    $$\lim_{n\to\infty}\epsilon_n^i\beta_n^i=\lambda_i\neq 0$$ exists in
    the space $\CM\CL(\Sigma)$ of measured lamination. We may also assume
    without loss of generality that $\epsilon_n^k=\epsilon_n'$ and thus
    that $\lambda_k=\lambda'$. The claim will follow when we show that
    $$\iota(\lambda,\lambda_1) = \iota(\lambda_1,\lambda_2) =
    \iota(\lambda_2,\lambda_3) = \dots = \iota(\lambda_{k-1},\lambda_k) =
    0.$$

    To do so, first note that 
    \begin{equation} \label{equ silly}
      \begin{aligned}
      \iota(\lambda,\lambda_1)
      &=\lim_n \epsilon_n \cdot\epsilon^1_n \cdot \iota(\phi_n(\gamma_0),
      \beta_n^1) \\
      &=\lim_n \epsilon_n \cdot \epsilon^1_n \cdot \iota(\phi_n(\gamma_0),
      \phi_n(\alpha_0)) \\
      &=\lim_n\epsilon_n\cdot\epsilon^1_n\cdot\iota(\gamma_0,\alpha_0)=0,
      \end{aligned}
    \end{equation}
    where the last equality follows from the fact that the sequence
    $(\epsilon_n^1)$ is bounded while $(\epsilon_n)$ tends to $0$. 
    The proof of the other equalities is even simpler: since the curves
    $\beta_n^i$ and $\beta_n^{i+1}$ are disjoint for all $n$ and $i$ we
    have $$\iota(\lambda_i,\lambda_{i+1}) = \lim_n\epsilon_n^i \cdot
    \epsilon_n^{i+1} \cdot \iota(\beta_n^i,\beta_n^{i+1})=0.$$ This
    finishes the proof of the lemma. \qedhere
    
  \end{proof}

  \begin{rmk} \label{rem silly}
    A key point in the proof of the Lemma \ref{lem silly lemma} is that
    there is a chain of simple multicurves from $\phi_n(\alpha_0)$ to
    $\psi_n(\alpha)$ such that consecutive curves are disjoint. Such a chain
    does not exist for the once-punctured torus $\Sigma_{1,1}$ or the
    4-times punctured sphere $\Sigma_{0,4}$.
    For these surfaces, we can find a chain of simple
    multicurves
    $$\phi_n(\alpha_0)=\beta_n^1,\beta_n^2,\dots,\beta_n^{k}=\psi_n(\alpha)$$
    with 
    \begin{displaymath}
      \iota(\beta_n^i,\beta_n^{i+1})
      = 
      \left\{ \begin{array}{ll}
        1 & \textrm{for } \Sigma_{1,1} \\
        2 & \textrm{for } \Sigma_{0,4},  
      \end{array}\right.
      i=1,\ldots,k-1.
    \end{displaymath}
    Let $\lambda = \lim_n \epsilon_n \phi_n(\beta_n)$ and $\lambda_i =
    \lim_n \epsilon_n^i \beta_n^i$. If we assume that $\lambda$ is uniquely
    ergodic, then we can still derive the conclusion of the lemma: namely,
    the chain \[ \lambda = \lambda_0,\lambda_1,\ldots,\lambda_k =
    \lambda'\] satisfies $\iota(\lambda_i,\lambda_{i+1}) = 0$. To see this,
    note that $\epsilon_n \to 0$ and that each $(\epsilon_n^i)$ is bounded.
    By the same argument as in Equation \ref{equ silly},
    $\iota(\lambda,\lambda_1) = 0$ since $\epsilon_n \to 0$. Now since
    $\lambda$ is uniquely ergodic, $\iota(\lambda,\lambda_1) =0$ implies
    $\lambda_1$ is a multiple of $\lambda$. In particular, $\lambda_1$ is
    uniquely ergodic, and thus the sequence $\epsilon_n^1 \to 0$. Using
    this, we get that $$\iota(\lambda_1,\lambda_2) = \lim_n\epsilon_n^1
    \cdot \epsilon_n^2 \cdot \iota(\gamma_n^1,\gamma_n^2)=0.$$ By the same
    argument as above, $\lambda_2$ is uniquely ergodic and hence
    $\epsilon_n^2 \to 0$. Continuing this way yields the result.
    
  \end{rmk}

  We are ready to prove the proposition. 

  %Still denoting by $\alpha_0$ the base point of the curve complex, let
  %$\Lambda(\phi)=\Lambda_{\alpha_0}(\phi)$ be the closure in the space of
  %projective measured lamination of the $C(\phi)$-orbit of $\alpha_0$,
  %that is: $$\Lambda(\phi)=\overline{\{\psi(\alpha_0)\text{ with }\psi\in
  %C(\phi)\}},$$ where the closure is taken in $\RP\CM\CL(\Sigma)$. Also,
  %for a set $K\subset\RP\CM\CL(S)$ define inductively $K^{(0)}=K$ and
  %$$K^{(n+1)}=\{\lambda\in\CM\CL\text{ for which there is }\mu\in
  %K^{(n)}\text{ with }\iota(\lambda,\mu)=0\}.$$ With this notation 

  \begin{proof}[Proof of Proposition \ref{kor silly kor}]
    Take for all $n$ some $\psi_n\in C(\phi)$ with
    $d_{\rel}(\phi_n,\psi_n)\le L$. Let $\alpha$ be any simple multicurve
    and let $\beta=\phi(\alpha)$. Compactness of $\RP\CM\CL(\Sigma)$
    implies that, up to passing to a subsequence, we may assume that the
    limits $$\lambda'=\lim_n\psi_n(\alpha) \quad \text{and} \quad \lambda''
    = \lim_n \psi_n(\beta)$$ exist in $\RP\CC(\Sigma)$. 
    
    From Lemma \ref{lem silly lemma} and from Remark \ref{rem silly}, there
    is a chain of measure laminations
    $$\lambda=\lambda_0,\lambda_1,\dots,\lambda_k = \lambda'$$ with
    $\iota(\lambda_i,\lambda_{i+1})=0$ for $i=1,\dots,k-1$. There is a
    similar chain from $\lambda$ to $\lambda''$. 

    Recall now that $\lambda_0=\lambda$ is uniquely ergodic. Since
    $\iota(\lambda_0,\lambda_1)=0$, we get that $\lambda_1$ is a multiple of
    $\lambda$ and thus uniquely ergodic. Then, since
    $\iota(\lambda_1,\lambda_2)=0$, we get that $\lambda_2$ is a multiple
    of $\lambda_1$ and thus of $\lambda$ and uniquely ergodic and so on.
    Iteratively we get that $\lambda'$ is a multiple of $\lambda$. Using
    the chain from $\lambda$ to $\lambda''$, we also get that $\lambda''$
    is a multiple of $\lambda$.
    
    Finally, since $\beta=\phi(\alpha)$ and $\psi_n \in C(\phi)$, we have
    that, projectively, $$\phi(\lambda') = \lim_{n\to\infty}
    \phi(\psi_n(\alpha)) = \lim_{n\to\infty}\psi_n(\phi(\alpha)) =
    \lim_{n\to\infty}\psi_n(\beta)=\lambda''.$$ This implies that $\lambda$
    is projectively fixed by $\phi$, so $\phi(\lambda)$ is a multiple of
    $\lambda$ as claimed.
  \end{proof}

\section{A technical result}

  The reason why we stressed earlier that $\CC(\Sigma)$ is metrizable and
  locally compact is that these are the properties needed to work
  as customary with the weak-*-topology on the space of
  measures\footnote{This is also the reason why we didn't encourage the
  reader to think of currents as measures, because it is a well-established
  fact that thinking of "the weak-*-topology on the space of measures on
  the space of measures endowed with the weak-*-topology" leads the
  unprepared reader to tremors, shaking and cold sweats.} on $\CC(\Sigma)$.
  In fact, to establish Theorem \ref{thm-counting curves} we will prove
  that the measures

  \begin{equation}\label{eq-gato}
    m^\CR_{\gamma_0,L}=\frac 1{L^{6g-6+2r}}\sum_{\phi\in\CR}\delta_{\frac
    1L\phi(\gamma_0)}
  \end{equation}
  converge when $L\to\infty$ to the trivial measure. Here we consider the
  weighted multicurve $\frac 1L\phi(\gamma_0)$ as a current and denote by
  $\delta_{\frac 1L\phi(\gamma_0)}$ the Dirac measure on $\CC(\Sigma)$
  centered therein.

  In \cite{ES-GAFA,Hugo,EU,RS} we considered a closely related family of
  measures and proved that the limit
  \begin{equation}\label{eq-maryam limit}
    C\cdot\mathfrak{m}_{\Thu}=\lim_{L\to\infty}\frac
    1{L^{6g-6+2r}}\sum_{\phi\in\Map(\Sigma)}\delta_{\frac 1L\phi(\gamma_0)}.
  \end{equation}
  exists (see also \cite{Book}). Here $C=C(\gamma_0)$ is a positive real
  number and $\mathfrak{m}_{\Thu}$ is the Thurston measure on
  $\CC(\Sigma)$. Recall that the Thurston measure is a Radon measure
  supported on the space $\CM\CL(\Sigma)$ of measured laminations. The
  Thurston measure can be constructed either as a scaling limit
  \cite{Maryam1,Book} or using the symplectic structure on
  $\CM\CL(\Sigma)$. See \cite{MT} for a discussion of both points of view.

  We only need the following facts about the Thurston measure. The action
  of $$\Map(\Sigma)\actson(\CM\CL(\Sigma),\mathfrak{m}_{\Thu})$$ preserves
  the measure and is ergodic \cite{Masur}. The action is also {\em almost
  free} in the sense that the fixed point set of every non-central element
  in $\Map(\Sigma)$ has vanishing Thurston measure (since the fixed point
  set has lower dimension). Note that the central elements of
  $\Map(\Sigma)$ act trivially on $\CM\CL(\Sigma)$. 

  In this section we prove:

  \begin{prop}\label{lem key lem counting}
    Let $\gamma_0\subset\Sigma$ be a filling multicurve. The family of
    measures $\big( m^{\CR}_{\gamma_0,L}\big)_{L\ge 1}$ is precompact with
    respect to the weak-*-topology on the space of Radon measures on
    $\CC(\Sigma)$. Moreover for any sequence $L_n\to\infty$ such that the
    limit $$\mathfrak{m}=\lim_{n\to\infty}m^\CR_{\gamma_0,L_n}$$ exists,
    one has that $$\sum_{\phi\in\Map(\Sigma)}\phi_*\mathfrak{m}\le
    C\cdot\mathfrak{m}_{\Thu}$$ where $C$ is as in \eqref{eq-maryam limit}.
  \end{prop}

  We start by proving that the family of measures in Proposition \ref{lem key lem
  counting} is precompact and that any limit must be uniformly continuous with respect to $\mathfrak{m}_{\Thu}$. 

  \begin{lemma}\label{lem-measures precompact}
    The family of measures $\big( m^\CR_{\gamma_0,L} \big)_{L\ge 1}$ is
    precompact with respect to the weak-*-topology on the space Radon
    measures on $\CC(\Sigma)$. Moreover, any accumulation point is
    absolutely continuous with respect to the Thurston measure.
  \end{lemma}

  \begin{proof}
    The measure $m^\CR_{\gamma_0,L}$ is bounded from above for all $L$ by
    the measure 
    \begin{equation}\label{eq-original}
      m_{\gamma_0,L}=\frac
      1{L^{6g-6+2r}}\sum_{\phi\in\Map(\Sigma)}\delta_{\frac
      1L\phi(\gamma_0)}.
    \end{equation}
    From the existence of the limit \eqref{eq-maryam limit} we get
    \begin{equation}\label{eq-I am hungry}
      \limsup\int fd m^\CR_{\gamma_0,L}\le\limsup\int fd
      m_{\gamma_0,L}=C\cdot\int fd\mathfrak{m}_{\Thu}<\infty
    \end{equation}
    for every continuous function $f:\CC(\Sigma)\to\BR$ with compact
    support. This implies that the family $( m_{\gamma_0,L}^\CR)_{L\ge 1}$
    is bounded and thus precompact in the weak-*-topology. Moreover, \eqref{eq-I am hungry}
    implies that any accumulation point of $m^\CR_{\gamma_0,L}$ is bounded
    from above by $C\cdot\mathfrak{m}_{\Thu}$ and hence is absolutely
    continuous to the Thurston measure, as we had claimed.
  \end{proof}

  Note that the same argument also proves that both families
  $$m^{\CI_k}_{\gamma_0,L}=\frac
  1{L^{6g-6+2r}}\sum_{\phi\in\CI_k}\delta_{\frac 1L\phi(\gamma_0)} \quad
  \text{and} \quad m^{\calD_k}_{\gamma_0,L}=\frac
  1{L^{6g-6+2r}}\sum_{\phi\in\calD_k}\delta_{\frac 1L\phi(\gamma_0)}$$ are
  precompact and that any limit when $L\to\infty$ is absolutely continuous
  with respect to the Thurston measure. Here $\CI_k$ and $\calD_k$ are, as
  before, the subsets of $\CR$ consisting of $k$-isolated points and
  $k$-dense points, respectively.

  We can from now on fix a sequence $(L_n)$ with $L_n\to\infty$ such that
  the following limits all exist:
  \begin{equation}\label{eq-running out of names 1}
  \begin{split}
    \mathfrak{m}_{\gamma_0}&=\lim_{n\to\infty}m^\CR_{\gamma_0,L_n},\\
    \mathfrak{m}^{\CI_k}_{\gamma_0}&=\lim_{n\to\infty}m^{\CI_k}_{\gamma_0,L_n},\text{
      and }\\
    \mathfrak{m}^{\calD_k}_{\gamma_0}&=\lim_{n\to\infty}m^{\calD_k}_{\gamma_0,L_n}.
  \end{split}
  \end{equation}
  Since $\CR$ is the disjoint union of $\CI_k$ and $\calD_k$ they
  automatically satisfy that
  $$\mathfrak{m}_{\gamma_0} = \mathfrak{m}_{\gamma_0}^{\CI_k} +
  \mathfrak{m}_{\gamma_0}^{\calD_k}.$$
  Our next goal is to prove that the second of these limits is $0$:

  \begin{lemma}\label{lem-k-discrete measure 0}
  We have $\mathfrak{m}_{\gamma_0}^{\calD_k}=0$.
  \end{lemma}
  \begin{proof}
    By Maher's Theorem \ref{thm-maher} it is enough to prove that, for any
    non-central $\phi_0\in\Map(\Sigma)$ and any $R\ge 0$, the trivial
    measure is the only accumulation point when $L\to\infty$ of the family
    of measures

    $$m_{\gamma_0,L}^{\CN}=\frac
    1{L^{6g-6+2r}}\sum_{\phi\in\CN_{\rel}(C(\phi_0),R)}\delta_{\frac
    1L\phi(\gamma_0)}.$$
    Well, each $m_{\gamma_0,L}^{\CN}$ is bounded by the measure
    $m_{\gamma_0, L}$ given by \eqref{eq-original} and hence any such
    accumulation point
    $\mathfrak{m}'=\lim_{n\to\infty}m_{\gamma_0,L_n}^{\CN}$ is bounded by
    $C\cdot\mathfrak{m}_{\Thu}$ by \eqref{eq-maryam limit}. The claim will
    then follow when we say that the support of $\mathfrak{m}'$ is
    contained in a set of vanishing Thurston measure.

    First, the support of the limiting measure $\mathfrak{m}'$ is contained
    in the set of accumulation points of sequences $(x_n)$ where $x_n$ is
    in the support of $m_{\gamma_0,L_n}^{\CN}$, that is, a multiple of
    $\phi_n(\gamma_0)$ for some $\phi_n\in\CN_{\rel}(C(\phi_0),R)$. On
    the other hand, since the set of uniquely ergodic lamination has full
    $\mathfrak{m}_{\Thu}$-measure \cite{Masur-uniquely ergodic}, we also
    get that $\mathfrak{m}'$ is supported by uniquely ergodic laminations.
    It thus follows from Proposition \ref{kor silly kor} that
    $\mathfrak{m}'$ is supported by the set of measured laminations
    projectively fixed by $\phi_0$. Since this set has vanishing
    $\mathfrak{m}_{\Thu}$-measure we get that $\mathfrak{m}'$ is trivial,
    as we needed to prove.
  \end{proof}

  As a final step towards the proof of Proposition \ref{lem key lem
  counting} we establish an equivariance property for the limits of the
  measures $m_{\gamma_0,L}^\CR$:

  \begin{lemma}\label{lem-equivariance}
    We have
    $\phi_*(\mathfrak{m}_{\gamma_0})=\lim_{n\to\infty}m^\CR_{\phi(\gamma_0),L_n}$
    for all $\phi\in\Map(\Sigma)$.
  \end{lemma}
  \begin{proof}
    Noting that the set $\CR$ is closed under conjugation we get that
    $\CR\phi=\phi\CR$. This means that
    \begin{align*}
      m^\CR_{\phi(\gamma_0),L_n}
      &=\frac 1{L^{6g-6+2r}}\sum_{\psi\in\CR}\delta_{\frac
      1L\psi\phi(\gamma_0)}=\frac
      1{L^{6g-6+2r}}\sum_{\psi\in\CR\phi}\delta_{\frac 1L\psi(\gamma_0)}\\
      &=\frac 1{L^{6g-6+2r}}\sum_{\psi\in\phi\CR}\delta_{\frac
      1L\psi(\gamma_0)}=\frac 1{L^{6g-6+2r}}\sum_{\psi\in\CR}\delta_{\frac
      1L\phi\psi(\gamma_0)}\\
      &=\frac 1{L^{6g-6+2r}}\sum_{\psi\in\CR}\phi_*\left(\delta_{\frac
      1L\psi(\gamma_0)}\right)=\phi_*\left(\frac
      1{L^{6g-6+2r}}\sum_{\psi\in\CR}\delta_{\frac
      1L\psi(\gamma_0)}\right)\\
      &=\phi_*\left(m^\CR_{\gamma_0,L_n}\right)
    \end{align*}
    The claim follows now from \eqref{eq-running out of names 1} and the
    continuity of the action of $\Map(\Sigma)$ on the space of currents.
  \end{proof}

  We are ready to prove the proposition:

  \begin{proof}[Proof of Proposition \ref{lem key lem counting}]

    Recall that Lemma \ref{lem-measures precompact} asserts that the given
    family of measures is precompact and hence we can assume that we are
    given a sequence $(L_n)$ with $L_n\to\infty$ such that the limit 
    $$\mathfrak{m}_{\gamma_0}=\lim_{n\to\infty}m^\CR_{\gamma_0,L_n}$$
    exists. To prove Proposition \ref{lem key lem counting} it will suffice
    to show, with $C$ as in \eqref{eq-maryam limit}, that for every finite
    set $\CZ\subset\Map(\Sigma)$ we have
    $$\sum_{\phi\in\CZ}\phi_*(\mathfrak{m}_{\gamma_0})\le
    C\cdot\mathfrak{m}_{\Thu}.$$
    Fixing such a finite set $\CZ$ choose 
    %$$k>2\cdot\max\{d_\CG(\phi,\phi')\text{ where }\phi,\phi\text{ are
    %distinct elements in }\CZ\}.$$
    $$k>2\cdot\max\{d_\CG(\text{id},\phi) \text{ where } \phi \in \CZ\}.$$
    Lemma \ref{lem-equivariance} and Lemma \ref{lem-k-discrete measure 0}
    imply, respectively, the first and last of the following equalities:

    $$\sum_{\phi\in\CZ}\phi_*(\mathfrak{m}_{\gamma_0}) =
    \sum_{\phi\in\CZ}\lim_{n\to\infty}m^\CR_{\phi(\gamma_0),L_n} =
    \lim_{n\to\infty}\sum_{\phi\in\CZ}m^\CR_{\phi(\gamma_0),L_n} =
    \lim_{n\to\infty}\sum_{\phi\in\CZ}m^{\CI_k}_{\phi(\gamma_0),L_n}.$$
    Moreover, from the choice of $k$ we get that
    $\CI_k\phi\cap\CI_k\phi'=\emptyset$ for any two distinct
    $\phi,\phi'\in\CZ$ and we can thus rewrite
    \begin{align*}
      \sum_{\phi\in\CZ}m^{\CI_k}_{\phi(\gamma_0),L_n} &=\frac
      1{L_n^{6g-6+2r}}\sum_{\phi\in\CZ}\sum_{\psi\in\CI_k}\delta_{\frac
      {1}{L_n} \psi\phi(\gamma_0)} =\frac
      1{L_n^{6g-6+2r}}\sum_{\phi\in\CZ}\sum_{\psi\in\CI_k\phi}\delta_{\frac
      {1}{L_n}\psi(\gamma_0)}\\ &=\frac
      1{{L_n}^{6g-6+2r}}\sum_{\psi \in \underset{\phi\in\CZ}{\bigcup}\CI_k\phi}\delta_{\frac
      {1}{L_n}\psi(\gamma_0)}.
    \end{align*}
    It thus follows that
    $$\sum_{\phi\in\CZ}m^{\CI_k}_{\phi(\gamma_0),L_n}\le\frac
    1{L_n^{6g-6+2r}}\sum_{\psi\in\Map(\Sigma)}\delta_{\frac
    {1}{L_n}\psi(\gamma_0)}$$ and hence that
    $$\sum_{\phi\in\CZ}\phi_*(\mathfrak{m}_{\gamma_0})\le\lim_{n\to\infty}\frac
    1{L_n^{6g-6+2r}}\sum_{\phi\in\Map(\Sigma)}\delta_{\frac
    {1}{L_n}\phi(\gamma_0)}=C\cdot\mathfrak{m}_{\Thu}.$$
    We are done.
  \end{proof}

\section{Proofs of the theorems}

  We are now ready to prove the main results. 

  \begin{named}{Theorem \ref{thm-counting curves}}
    Let $\CR \subset \Map(\Sigma)$ be the set of non pseudo-Anosov mapping
    classes and let $\gamma_0\subset\Sigma$ be a filling multicurve. Then
    we have $$\lim_{L\to\infty}\frac{\vert\{\phi\in\CR\text{ with
    }F(\phi(\gamma_0))\le L\}\vert}{L^{6g-6+2r}}=0$$ for every continuous
    homogenous function $F:\CC(\Sigma)\to\BR_{\ge 0}$ which, for every
    $A\subset\Sigma$ compact, is proper when restricted on the set
    $\CC_A(\Sigma)$ of currents supported by $A$.
  \end{named}

  \begin{proof}
    The claim will follow easily once we prove that
    \begin{equation}\label{eq-this is a pain}
    \lim_{L\to\infty}m_{\gamma_0,L}^\CR=0
    \end{equation}
    with $m_{\gamma_0,L}^\CR$ as in \eqref{eq-gato}. Since this family of
    measures is precompact by Proposition \ref{lem key lem counting}, it
    suffices to prove that $0$ is the only accumulation point when
    $L\to\infty$. So let $(L_n)$ be a sequence tending to $\infty$ and
    such that the limit
    $$\mathfrak{m}_{\gamma_0}=\lim_{n\to\infty}m_{\gamma_0,L_n}^\CR$$
    exists. By Lemma \ref{lem-measures precompact}
    $\mathfrak{m}_{\gamma_0}$ is absolutely continuous with
    respect to $\mathfrak{m}_{\Thu}$. This means that there is a function
    (the Radon-Nikodym derivative)
    $\kappa:\CC(\Sigma)\to\BR_{\ge 0}$ with the property that
    $$\int_{\CC(\Sigma)}
    f(\zeta)d\mathfrak{m}_{\gamma_0}(\zeta)=\int_{\CC(\Sigma)}
    f(\zeta)\cdot\kappa(\zeta)d\mathfrak{m}_{\Thu}(\zeta)$$ for any
    continuous compactly supported function $f$ on the space of currents.

    Proposition \ref{lem key lem counting} asserts that the measure
    $\sum_{\phi\in\Map(\Sigma)}\phi_*\mathfrak{m}_{\gamma_0}$ is not only
    finite, but actually bounded by a multiple $C\cdot\mathfrak{m}_{\Thu}$
    of the Thurston measure. In terms of the function $\kappa$, this
    implies that 
    \begin{equation}\label{eq almost done}
      \sum_{\phi\in\Map(\Sigma)}\kappa(\phi(\zeta))\le C\text{ for
      }\mathfrak{m}_{\Thu}\text{-almost every }\zeta\in\CC(\Sigma).
    \end{equation}
    We claim that this implies that $\kappa(\zeta)=0$ almost surely: 

    \begin{claim}
    $\kappa(\zeta)=0$ almost surely with respect to the Thurston measure.
    \end{claim}

    In a nutshell, the claim follows from the fact that ergodic actions of
    discrete groups on non-atomic measure spaces are recurrent (the
    condition on the measure being non-atomic is just there to rule out 
    actions with only one orbit). In any case, we give a direct
    argument to prove the claim:

    \begin{proof}[Proof of the Claim]
      If the claim fails to be true, then there is a positive
      $\mathfrak{m}_{\Thu}$-measure set $U\subset\CC(\Sigma)$ with
      $\kappa(\zeta)\ge\epsilon>0$ for every $\zeta\in U$. Noting that the
      action $\Map(\Sigma)\actson\CC(\Sigma)$ is almost free we get from
      \eqref{eq almost done} that, for almost every $\zeta\in U$,
      $$\#\{\phi\in\Map(\Sigma)\text{ with }\phi(\zeta)\in U\}\le\frac
      C\epsilon.$$
      %$$\#\{\phi\in\Map(\Sigma)\text{ with }\phi(\zeta)\in U\}\le\frac
      %C\epsilon\cdot\vert C(\Map(\Sigma))\vert$$
      It follows that there is a set $V\subset U$ of positive Thurston measure such that
      the set $$\CZ=\{\phi\in\Map(\Sigma)\text{ with }\phi(V)\cap
      U\neq\emptyset\}$$ is finite. Now, since the action is essentially
      free we can in fact find $W\subset V$ of positive Thurston  measure
      with $W\cap\phi(W)=\emptyset$ for all $\phi\notin C(\Map(\Sigma))$.
      This contradicts the ergodicity of the action of the mapping class
      group on $(\CM\CL(\Sigma),\mathfrak{m}_{\Thu})$.
  \end{proof}

  The claim implies that the limiting measure vanishes, that is
  $\mathfrak{m}_{\gamma_0}=0$, establishing \eqref{eq-this is a pain}.
  We can now conclude the proof: let
  $F:\CC(\Sigma)\to\BR_{\ge 0}$ be as in the statement and note that
  \begin{align*}
    \frac{\vert\{\gamma\in\CR\cdot\gamma_0\text{ with }F(\gamma)\le
    L\}\vert}{L^{6g-6+2r}} &\le\frac{\vert\{\phi\in\CR\text{ with }F \left(
    \frac 1L\phi(\gamma_0) \right)\le 1\}\vert}{L^{6g-6+2r}}\\
    &=m^{\CR}_{\gamma_0,L}(\{F(\cdot)\le 1\})
  \end{align*}
  and by \eqref{eq-this is a pain} together with the fact that 
  $\{F(\cdot)\le 1\}$ is compact we have that
  $$\lim_{L\to\infty}m^{\CR}_{\gamma_0,L}(\{F(\cdot)\le 1\})=0. \qedhere$$
  \end{proof}
 
  Finally, we prove Theorem \ref{thm-distances}. Recall that for $f \in
  \Diff(\Sigma)$, $K(f)$ and $\Lip(f)$ are respectively the quasi-conformal
  distortion and Lipschitz constant of $f$.

  \begin{named}{Theorem \ref{thm-distances}}
    The set of pseudo-Anosov mapping classes is generic with respect to any
    one of the functions:
    \begin{enumerate}
      \item $\rho_{K}(\phi)=\inf\{K(f)\text{ where }f\in\Diff(\Sigma)\text{
      represents }\phi\}$. 
      \item $\rho_{\Lip}(\phi)=\inf\{\Lip(f)\text{ where }f\in\Diff(\Sigma)\text{
      represents }\phi\}$.
      \item $\rho_{\sigma,\eta}(\phi)=\iota(\phi(\sigma),\eta)$, where
        $\sigma$ and $\eta$ are filling multicurves. 
    \end{enumerate}
  \end{named}

  \begin{proof}
    We start by proving that the set of pseudo-Anosov mapping classes is
    $\rho_{\sigma,\eta}$-generic for filling multicurves $\sigma$ and
    $\eta$. Well, the function
    $$\CC(\Sigma)\to\BR_{\ge 0}, \ \lambda\mapsto\iota(\lambda,\eta)$$
    is continuous and proper on the set $\CC_A(\Sigma)$ of currents
    supported by compact sets $A\subset\Sigma$. We thus get from Theorem
    \ref{thm-counting curves} that
    \begin{equation}\label{eq-running out of names}
      \lim_{L\to\infty}\frac{\vert\{\phi\in\CR\text{ with
      }\iota(\phi(\sigma),\eta)\le L\}\vert}{L^{6g-6+2r}}=0
    \end{equation}
    On the other hand we get from \cite{Sapir} or \cite{Maryam2} (see also
    \cite{ES-announcement,Book}) that 
    \begin{equation}\label{eq-running out of names2}
      \liminf_{L\to\infty}\frac{\vert\{\phi\in\Map(\Sigma)\text{ with
      }\iota(\phi(\sigma),\eta)\le
      L\}\vert}{L^{6g-6+2r}}=\const(\sigma,\eta)>0.
    \end{equation}
    Since $\rho_{\sigma,\eta}(\phi)=\iota(\phi(\sigma),\eta)$ we get from
    \eqref{eq-running out of names} and \eqref{eq-running out of names2}
    that
    $$\lim_{L\to\infty}\frac{\vert\{\phi\in\CR\text{ with
    }\rho_{\sigma,\eta}(\phi)\le L\}\vert}{\vert\{\phi\in\Map(\Sigma)\text{
      with }\rho_{\sigma,\eta}(\phi)\le L\}\vert}=0.$$
    This shows the set of pseudo-Anosov mapping classes is generic with
    respect to $\rho_{\sigma,\eta}$.

    We consider now genericity with respect to $\rho_{\Lip}$. Fix once and for all 
    a filling multicurve $\sigma$. Although it
    does not really matter, we could for example assume that $\sigma$ is a
    marking in the sense of \cite{Masur-Minsky CC2}. We need the following fact:

    \begin{fact}\label{fact1}
      There is $C=C(\Sigma,\sigma)\ge 1$ with 
      $$\frac 1C\cdot\rho_{\Lip}(\phi)\le\rho_{\sigma,\sigma}(\phi)\le
      C\cdot\rho_{\Lip}(\phi)$$
      for all $\phi\in\Map(\Sigma)$.
    \end{fact}

    Fact \ref{fact1} is well known but, for the convenience of the reader,
    we will comment on its proof once we are done with Theorem
    \ref{thm-distances}. From Fact \ref{fact1} we get that
    \begin{align*}
      \vert\{\phi\in\CR\text{ with }\rho_{\Lip}(\phi)\le
      L\}\vert&\le\vert\{\phi\in\CR\text{ with
      }\rho_{\sigma,\sigma}(\phi)\le C L\}\vert\\
      \vert\{\phi\in\Map(\Sigma)\text{ with }\rho_{\Lip}(\phi)\le
      L\}\vert&\ge \left\vert \Big\{\phi\in\Map(\Sigma)\text{ with
      }\rho_{\sigma,\sigma}(\phi)\le \frac LC \Big\} \right\vert.
    \end{align*}
    We thus get from \eqref{eq-running out of names} and \eqref{eq-running
    out of names2} that 
    $$\lim_{L\to\infty}\frac{\vert\{\phi\in\CR\text{ with
    }\rho_{\Lip}(\phi)\le L\}\vert}{\vert\{\phi\in\Map(\Sigma)\text{ with
    }\rho_{\Lip}(\phi)\le L\}\vert}=0,$$
    as we had claimed.

    The genericity with respect to $\rho_K$ follows by the
    same argument when we replace Fact \ref{fact1} by the following also well-known
    fact:

    \begin{fact}\label{fact2}
      There is $C=C(\Sigma)\ge 1$ with 
      $$\frac 1C\cdot\rho_{\Lip}(\phi)^2\le\rho_K(\phi)\le
      C\cdot\rho_{\Lip}(\phi)^2$$
      for all $\phi\in\Map(\Sigma)$.
    \end{fact}

    We have proved Theorem \ref{thm-distances}.
  \end{proof}

  We comment now on the proofs of the two facts used in the proof
  above. By properties (7) and (8) of the space of currents, we have that
  for any other filling multicurve $\sigma'$ there is a constant 
  $C_1=C_1(\Sigma,\sigma,\sigma')$ with  
  \begin{equation} \label{LenInt}
    \frac{1}{C_1} \iota(\sigma,\phi(\sigma)) \le \ell_\Sigma(\phi(\sigma'))
    \le C_1 \iota(\sigma,\phi(\sigma)),
  \end{equation}
  where $\ell_\Sigma(\cdot)$ is the hyperbolic length function.
  Choosing $\sigma'$ to be a \emph{short marking} in the sense of
  \cite{LRT}, we get from Theorem 4.1 in that paper that there is a constant
  $C_2 = C_2(\Sigma,\sigma')$ such that 
  \begin{equation} \label{LipLen}
    \frac{1}{C_2} \ell_\Sigma(\sigma') \le \rho_{\Lip}(\phi) \le C_2
    \ell_\Sigma(\sigma') 
  \end{equation}
  Fact \ref{fact1} follows, with $C=C_1 \cdot C_2$, from these two
  inequalities.
 
  A similar argument, replacing results from \cite{LRT} by results from
  \cite{Rafi}, yields Fact 2. Alternatively one can directly refer to
  Theorem B in \cite{Choi-Rafi}.

  For the reader who feels cheated by a proof which only consists of a
  sequence of references, we sketch a more direct proof of Fact 1 and Fact
  2. Suppose $\Sigma$ is closed. By the Arzel\'a-Ascoli theorem, there is a
  Lipschitz map $f$ on $\Sigma$ representing $\phi$ with $L_f =
  \rho_{\Lip}(\phi)$. By Teichm\"uller's theorem, there is a unit-area
  quadratic differential $q$ on $\Sigma$ and a map $g$ representing $\phi$,
  such that $\rho_K(\phi) = L_g^2$, where $L_g$ is the Lipschitz constant
  of $g$ with respect to the singular Euclidean metric induced by $q$.
  Moreover, $L_g$ is the minimal Lipschitz constant of all maps on $q$
  representing $\phi$. By compactness of $\Sigma$, the $q$--metric and the
  hyperbolic metric on $\Sigma$ are bilipschitz equivalent. By compactness
  of the space of unit-area quadratic differentials, this bilipschitz
  equivalence is uniform. Therefore, there is a constant $B$ depending only
  on $\Sigma$ such that $$\frac {1}{B} L_f \le L_g \le B L_f.$$ This
  obtains Fact \ref{fact2} with $C=B^2$.
  
  Let $\sigma$ be a filling multicurve which we realize by a $q$--geodesic.
  Because $\sigma$ is filling, it cannot be entirely $q$--vertical.
  Compactness of the space of unit-area quadratic differentials implies
  that in fact the horizontal length of $\sigma$ is a definite proportion
  of its total length. Under the map $g$, the $q$--horizontal direction
  gets stretched by the factor $L_g$, so the $q$--length of $\phi(\sigma)$
  grows proportionally to $L_g$. By comparing to the hyperbolic metric and
  using compactness of $\Sigma$ again, we get Equation (\ref{LipLen}) with
  $\sigma=\sigma'$. We still have (\ref{LenInt}) (with $\sigma=\sigma')$.
  This shows Fact \ref{fact1}.

  For the general case, losing compactness of $\Sigma$ means losing
  bilipschitz equivalence between the $q$--metric and the hyperbolic
  metric. However, the argument we just sketched can be modified to take
  care of this issue and we refer to the above listed references for the
  details. 
   
%section{References}

  \end{document}